\documentclass[review]{elsarticle}

\usepackage[bookmarksnumbered, plainpages, colorlinks]{hyperref}

\journal{Applied Mathematics Letters}

\usepackage{algorithm}
\usepackage{algorithmicx}
\usepackage{algpseudocode}
\usepackage{amsmath}
\usepackage{amssymb}
\usepackage{float}

\usepackage[top=2cm, bottom=2cm, left=2cm, right=2cm]{geometry}
\usepackage{algorithm}
\usepackage{algorithmicx}
\usepackage{algpseudocode}
\usepackage{amsmath}
\usepackage{amssymb}
\usepackage{float}

\usepackage{graphicx, latexsym, epsfig, amssymb, mathrsfs}
\usepackage{enumerate, xspace, amscd, amsfonts, graphicx, color}
\usepackage{amscd, verbatim, amsmath}
\usepackage{times}
\usepackage{amsthm}
\usepackage{subfigure}
\graphicspath{{figures/}}
\usepackage{booktabs}

\newtheorem{theorem}{Theorem}[section]

\newtheorem{remark}{Remark}[section]

\newtheorem{lemma}{Lemma}[section]

\numberwithin{equation}{section}

\begin{document}

\begin{frontmatter}

\title{Strong convergence of modified inertial Mann algorithms for nonexpansive mappings}


\author[mymainaddress]{Bing Tan}
\ead{bingtan72@gmail.com}

\author[mymainaddress]{Zheng Zhou}
\ead{zhouzheng2272@163.com}

\author[mymainaddress]{Songxiao Li\corref{mycorrespondingauthor}}
\cortext[mycorrespondingauthor]{Corresponding author}
\ead{jyulsx@163.com}

\address[mymainaddress]{Institute of Fundamental and Frontier Sciences, University of Electronic Science and Technology of China, Chengdu 611731, China}
\begin{abstract}
In this paper, we introduce two new modified inertial Mann Halpern and viscosity algorithms for solving fixed point problems. We establish  strong convergence theorems under some suitable conditions. Finally, our algorithms are applied to split feasibility problem, convex feasibility problem and location theory. The algorithms and results presented in this paper can summarize and improve corresponding results previously known in this area.
\end{abstract}

\begin{keyword}
Halpern algorithm  \sep Viscosity algorithm  \sep Inertial algorithm \sep  Nonexpansive mapping  \sep  Strong convergence
\MSC[2010] 47H05 \sep   49J40 \sep 90C52
\end{keyword}

\end{frontmatter}

\section{Introduction-Preliminaries}
Let $ C $ be a nonempty closed convex subset of a real Hilbert space $ H $.   A mapping $T: C \rightarrow C$  is said to be nonexpansive if $\|T x-T y\| \leq\|x-y\|$ for all $x, y \in C$. The set of fixed points of a mapping $T: C \rightarrow C$  is defined by $\operatorname{ Fix }(T):=\{x \in C: T x=x\}$. For  any $ x \in H$, ${P}_{C} x$ denotes the  metric projection of $ H $ onto $ C $, such that $ {P}_{C}(x) {:=}\operatorname{argmin}_{y \in C}\|x-y\| $.

In this paper, we consider the following fixed point problem: find $x^{*}\in C$, such that $T\left(x^{*}\right)=x^{*}$, where $T: C \rightarrow C$  is nonexpansive with $\operatorname{ Fix }(T) \neq \emptyset$. Approximation of fixed point problems with nonexpansive mappings has various specific applications, because many problems can be considered as fixed point problems with nonexpansive mappings. For instance, monotone variational inequalities, convex optimization problems, convex feasibility problems  and image restoration problems.  It is known that the Picard iteration algorithm may not converge. One way to overcome this difficulty is to use Mann’s iteration algorithm that produces a sequence $\left\{x_{n}\right\}$ via the following:
\begin{equation}\label{mann}
	x_{n+1}=\psi_{n} x_{n}+\left(1-\psi_{n}\right) T x_{n}, \quad n \ge 0,
\end{equation}
the iterative sequence $\left\{x_{n}\right\}$ defined by \eqref{mann} converges weakly to a fixed point of $ T $ provided that $\left\{\psi_{n}\right\} \subset(0,1)$   satisfies  $\sum_{n=0}^{\infty} \psi_{n}\left(1-\psi_{n}\right)=+\infty$.

In practical applications, many problems, such as, quantum physics and image reconstruction, are in  infinite dimensional spaces. To investigate these problems, norm convergence is usually preferable to the weak convergence. Therefore, modifying the Mann iteration algorithm to obtain strong convergence has been studied by many authors, see \cite{nakajo2003strong,kim2005strong,kim2006strong,marino2007weak,yao2008strong,takahashi2008strong,qin2009strong} and  the references therein.  In 2003, Nakajo and Takahashi \cite{nakajo2003strong} established strong convergence of the Mann iteration with the aid of projections. Indeed, they considered the following algorithm:
\begin{equation}\label{cq}
	\left\{\begin{array}{l}{y_{n}=\psi_{n} x_{n}+\left(1-\psi_{n}\right) T x_{n}}, \\ {C_{n}=\left\{u \in C:\left\|y_{n}-u\right\| \le\left\|x_{n}-u\right\|\right\}}, \\ {Q_{n}=\left\{u \in C:\left\langle x_{n}-u, x_{n}-x_{0}\right\rangle \le 0\right\}}, \\ {x_{n+1}={P}_{C_{n} \cap Q_{n}} x_{0}, \quad n \ge 0},\end{array}\right.
\end{equation}
where $\left\{\psi_{n}\right\} \subset[0,1)$, $ T $ is a nonexpansive mapping on $C$ and  ${P}_{C_{n} \cap Q_{n}}$ is the  metric projection from $C$ onto $C_{n} \cap Q_{n}$. This method is now referred as  the CQ algorithm.  For further research, see \cite{takahashi2008strong,qin2009convergence,zhang2009self,qin2010hybrid,van2017modified}. Recently, Kim and Xu \cite{kim2005strong} proposed the following modified Mann iteration algorithm based on the  Halpern iterative algorithm \cite{halpern1967fixed} and the Mann iteration algorithm:
\begin{equation}\label{Kim}
	\left\{\begin{array}{ll}  {y_{n}=\psi_{n} x_{n}+\left(1-\psi_{n}\right) T x_{n},}  \\ {x_{n+1}=\nu_{n} u+\left(1-\nu_{n}\right) y_{n}}, n \ge 0,\end{array}\right.
\end{equation}
where $u \in C$ is an arbitrary (but fixed) element in $C$. They obtained a strong convergence theorem  of iteration algorithm \eqref{Kim} as follows:
\begin{theorem}
	Let $ C $ be a closed convex subset of a uniformly smooth Banach space $ X $ and let $T: C \rightarrow C$ be a nonexpansive mapping with  $\operatorname{Fix}(T) \neq \emptyset$. Given a point $u \in C$ and given sequences $\left\{\psi_{n}\right\}$ and $\left\{\nu_{n}\right\}$ in $(0,1)$,  the
	following conditions are satisfied:
	\begin{enumerate}[(C1)]
		\item $\lim_{n \rightarrow \infty}\psi_{n} = 0,\; \sum_{n=0}^{\infty} \psi_{n}=\infty$ and $\sum_{n=0}^{\infty}\left|\psi_{n+1}-\psi_{n}\right|<\infty$;
		\item $\lim_{n \rightarrow \infty}\nu_{n} = 0,\; \sum_{n=0}^{\infty} \nu_{n}=\infty$ and $\sum_{n=0}^{\infty}\left|\nu_{n+1}-\nu_{n}\right|<\infty$.
	\end{enumerate}
	Then the sequence $\left\{x_{n}\right\}$ defined by \eqref{Kim} converges strongly to a fixed point of $T$.
\end{theorem}

Inspired by the result of  Kim and Xu \cite{kim2005strong}, Yao, Chen and Yao \cite{yao2008strong} introduced  a new modified Mann iteration algorithm by combines the viscosity approximation algorithm \cite{moudafi2000viscosity} and  the modified Mann iteration algorithm \cite{kim2005strong}. They  established  strong convergence in a uniformly smooth Banach space under some  fewer restrictions. It should be noted that there is no additional projection involved in \cite{kim2005strong} and \cite{yao2008strong}. For further research, see \cite{martinez2006strong,qin2007approximation,zhou2008convergence,ceng2008Mann,he2019optimal}.

In general, the convergence rate of Mann algorithm  is slow. Fast convergence of algorithm is required in many practical applications. In particular, an inertial type extrapolation was first proposed by Polyak \cite{polyak1964some} as an acceleration process. In recent years, some authors have constructed different fast iterative algorithms by inertial extrapolation techniques, such as, inertial Mann algorithms \cite{mainge2008convergence},  inertial forward-backward splitting algorithms \cite{lorenz2015inertial}, inertial extragradient algorithms \cite{liu2019strong,fan2019subgradient} and  fast iterative shrinkage-thresholding algorithm  (FISTA) \cite{beck2009fast}. In 2008, Mainge \cite{mainge2008convergence} introduced the following inertial Mann algorithm by unifying the inertial extrapolation and the Mann algorithm:
\begin{equation}\label{imann}
	\left\{\begin{array}{l}{w_{n}=x_{n}+\delta_{n}\left(x_{n}-x_{n-1}\right)}, \\ {x_{n+1}=\psi_{n} w_{n}+\left(1-\psi_{n}\right) T w_{n}},n\ge 0.\end{array}\right.
\end{equation}	
Then the iterative sequence $ \{x_n \} $ defined by \eqref{imann} converges weakly to a fixed point of $ T $ under some mild  assumptions.

Inspired and motivated by the works of Kim and Xu \cite{kim2005strong}, Yao, Chen and Yao \cite{yao2008strong} and Mainge \cite{mainge2008convergence},  we propose modified inertial Mann Halpern algorithm and modified inertial Mann viscosity algorithm, respectively. Strong convergence results are obtain under some mild conditions. Finally, our algorithms are applied to split feasibility problems, convex feasibility problems and location theory. Our algorithms and results  generalize and improve some corresponding previously known results.

Throughout this paper, we denote the strong and weak convergence of a sequence $\left\{x_{n}\right\}$ to a point $x \in H$ by $x_{n} \rightarrow x$ and $x_{n} \rightharpoonup x$, respectively. For each $x,y \in H$, we have the following facts.
\begin{enumerate}[(1)]
	\item $\|x+y\|^{2} \leq\|x\|^{2}+2\langle y, x+y\rangle $;
	\item $\|t x+(1-t) y\|^{2}=t\|x\|^{2}+(1-t)\|y\|^{2}-t(1-t)\|x-y\|^{2} , \quad \forall t\in\mathbb{R} $;
	\item $\langle {P}_{C} x-x, {P}_{C} x-y\rangle \le 0, \quad \forall y \in C $.
\end{enumerate}

\begin{lemma}\cite{bauschke2011convex}\label{lem:fix}
	Let $ C $ be a nonempty closed convex subset of a real Hilbert space  $ H $, $T: C \rightarrow {H}$  be a nonexpansive mapping. Let $\left\{x_{n}\right\}$ be a sequence
	in $ C $ and $ x\in H $ such that $x_{n} \rightharpoonup x$ and $T x_{n}-x_{n} \rightarrow 0$ as $n \rightarrow+\infty$. Then $x \in \operatorname{Fix}(T)$.
\end{lemma}

\begin{lemma}\cite{cholamjiak2018inertial}\label{lemma1}
	Assume $\left\{S_{n}\right\}$  is a sequence of nonnegative real numbers such
	that
	\[
	\begin{aligned}
	S_{n+1} \leq \left(1-\nu_{n}\right) S_{n}+\nu_{n}  \sigma_{n}, \forall n\ge 0,\; \text{and} \;  S_{n+1} \leq S_{n}-\eta_{n}+\pi_{n}, \forall n\ge 0,
	\end{aligned}
	\]
	where $\left\{\nu_{n}\right\}$  is a sequence in $ (0,1) $, $\left\{\eta_{n}\right\}$  is a sequence of nonnegative real numbers, $\left\{\sigma_{n}\right\}$ and $\left\{\pi_{n}\right\}$  are real sequences such that (i) $\sum_{n=0}^{\infty} \nu_{n}=\infty$; (ii) $\lim _{n \rightarrow \infty} \pi_{n}=0$; (iii)  $\lim _{k \rightarrow \infty} \eta_{n_{k}}=0$  implies $\limsup _{k \rightarrow \infty} \sigma_{n_{k}} \leq 0$  for any subsequence $\left\{\eta_{n_{k}}\right\}$ of $\left\{\eta_{n}\right\}$. Then $\lim _{n \rightarrow \infty} S_{n}=0$.
\end{lemma}

\section{Modified inertial Mann Halpern and viscosity algorithms}\label{sec3}
In this section, combining the idea of inertial with the Halpern algorithm and viscosity algorithm, respectively, we introduce two modified inertial Mann algorithms and  analyzes their convergence.
\begin{theorem}
	Let $ C $ be a nonempty closed convex subset of a real Hilbert space $ H $ and let $T: C \rightarrow C$  be a nonexpansive mapping with  $\operatorname{ Fix }(T) \neq \emptyset$.  Given a point  $u \in C$ and given two sequences $\left\{\psi_{n}\right\}$ and $\left\{\nu_{n}\right\}$ in $(0,1)$, the following conditions are satisfied:
	\begin{enumerate}[(D1)]
		\item $\lim_{n \rightarrow \infty}\nu_{n} = 0$ and  $\sum_{n=0}^{\infty} \nu_{n}=\infty$;
		\item $ \lim_{n \rightarrow\infty} \frac{\delta_{n}}{\nu_{n}}\left\|x_{n}-x_{n-1}\right\|=0 $.
	\end{enumerate}
	Set $ x_{-1},x_{0}\in C $ be arbitarily. Define a sequence $\left\{x_{n}\right\}$  by the following algorithm:
	\begin{equation}\label{MIMHA}
		\left\{\begin{array}{ll} w_{n}=x_{n} + \delta_{n}(x_{n}-x_{n-1}), \\ {y_{n}=\psi_{n} w_{n}+\left(1-\psi_{n}\right) T w_{n},}  \\ {x_{n+1}=\nu_{n} u+\left(1-\nu_{n}\right) y_{n}}, n \ge 0.\end{array}\right.
	\end{equation}
	Then the iterative sequence $\left\{x_{n}\right\}$ defined by \eqref{MIMHA} converges strongly to  $p=P_{\operatorname{Fix}(T)} u$.
\end{theorem}

\begin{proof}
	First we show that $\left\{x_{n}\right\}$ is bounded. Indeed, taking $p \in \operatorname{Fix}(T) $, we have
	\begin{equation}\label{eqe}
		\begin{aligned}\left\|x_{n+1}-p\right\| & \leq \nu_{n}\|u-p\|+\left(1-\nu_{n}\right)\left\|y_{n}-p\right\| \\ & \le \nu_{n}\|u-p\| + \left(1-\nu_{n}\right)\left(\psi_{n}\|w_{n}-p\| + (1-\psi_{n})\|Tw_{n}-p\|\right)  \\ & \le  \left(1-\nu_{n}\right)\left\|x_{n}-p\right\|+\nu_{n}\left\|u-p\right\|+\left(1-\nu_{n}\right)\delta_{n}\left\|x_{n}-x_{n-1}\right\|. \end{aligned}
	\end{equation}
	Let $ M:=2\max \left\{\|u-p\|, \sup _{n \ge 0}\frac{\left(1-\nu_{n}\right)\delta_{n}}{\nu_{n}}\left\|x_{n}-x_{n-1}\right\|\right\} $. Then \eqref{eqe} reducing to the following:
	\begin{equation}\label{eqdc}
		\left\|x_{n+1}-p\right\| \leq \left(1-\nu_{n}\right)\left\|x_{n}-p\right\| + \nu_{n} M \le \max\{\|x_{n}-p\|,M\}\le \cdots \le \max\{\|x_{0}-p\|,M\}.
	\end{equation}
	Combining condition (D2) and \eqref{eqdc}, we obtain that $\left\{x_{n}\right\}$ is bounded. So $ \{w_{n}\} $ and $\left\{y_{n}\right\}$ are also bounded. By the definition of $ y_{n}  $ in \eqref{MIMHA}, we have	
	\begin{equation}\label{eqdf}
		\begin{aligned}  \|y_{n}-p\|^{2} &= \psi_{n}\|w_{n}-p\|^{2}+\left(1-\psi_{n}\right)\| T w_{n}- p\left\|^{2}-\psi_{n}\left(1-\psi_{n}\right)\right\| T w_{n}-w_{n} \|^2 \\ &\le \|w_{n}-p\|^{2}-\psi_{n}\left(1-\psi_{n}\right)\| T w_{n}-w_{n} \|^{2}. \end{aligned}
	\end{equation}
	Therefore, from the definition of $ w_{n} $ and \eqref{eqdf}, we get
	\begin{equation}\label{eqec}
		\begin{aligned}
			\|x_{n+1}-p\|^2
			& = \|(1-\nu_{n})(y_{n}-p) + \nu_{n}(u-p)\|\\
			&\le \left(1-\nu_{n}\right)^2\left\|y_{n}-p\right\|^{2} + 2 \nu_{n}\langle u-p,  x_{n+1}-p\rangle \\
			&\le \left(1-\nu_{n}\right)\|w_{n}-p\|^{2}-\psi_{n}\left(1-\psi_{n}\right)\left(1-\nu_{n}\right)\| T w_{n}-w_{n} \|^{2}+ 2 \nu_{n}\langle u-p,  x_{n+1}-p\rangle \\
			&= \left(1-\nu_{n}\right)\left\|x_{n}-p\right\|^{2}+\delta_{n}^{2}\left(1-\nu_{n}\right)\left\|x_{n}-x_{n-1}\right\|^{2}+2 \delta_{n}\left(1-\nu_{n}\right)\langle x_{n}-x_{n-1},  x_{n}-p \rangle\\
			&\quad -\psi_{n}\left(1-\psi_{n}\right)\left(1-\nu_{n}\right)\| T w_{n}-w_{n} \|^{2}+ 2 \nu_{n}\langle u-p,  x_{n+1}-p\rangle. \\
		\end{aligned}
	\end{equation}
	For each $ n\ge 0 $, let
	\[
	\begin{aligned}
	S_n &= \left\|x_{n}-p\right\|^{2}, \;\pi_{n} = \nu_{n}\sigma_{n},\\
	\sigma_{n} &= \frac{\delta_{n}^{2}\left(1-\nu_{n}\right)}{\nu_{n}}\left\|x_{n}-x_{n-1}\right\|^{2}+\frac{2 \delta_{n}\left(1-\nu_{n}\right)}{\nu_{n}}\langle x_{n}-x_{n-1},  x_{n}-p \rangle + 2 \langle u-p,  x_{n+1}-p\rangle, \\
	\eta_{n} &= \psi_{n}\left(1-\psi_{n}\right)\left(1-\nu_{n}\right)\| T w_{n}-w_{n} \|^{2}.
	\end{aligned}
	\]
	Then \eqref{eqec} reduced to the following:
	\[
	\begin{aligned}
	S_{n+1} \leq \left(1-\nu_{n}\right) S_{n}+\nu_{n}  \sigma_{n}, \quad \text{and} \quad S_{n+1} \leq S_{n}-\eta_{n}+\pi_{n}.
	\end{aligned}
	\]
	From conditions (D1) and (D2), we obtain $ \sum_{n=0}^{\infty} \nu_{n} = \infty $ and $ \lim_{n \rightarrow\infty} \pi_{n} =0 $. In order to complete the proof, using Lemma \ref{lemma1}, it remains to show that $\lim _{k \rightarrow \infty} \eta_{n_{k}}=0$ implies $\limsup _{k \rightarrow \infty} \sigma_{n_{k}} \leq 0$  for any subsequence  $\left\{\eta_{n_{k}}\right\}$ of $\left\{\eta_{n}\right\}$. Let $\left\{\eta_{n_{k}}\right\}$ be a subsequence of $\left\{\eta_{n}\right\}$ such that $\lim _{k \rightarrow \infty} \eta_{n_{k}}=0$, which implies that $ \lim_{k \rightarrow \infty} \|Tw_{n_{k}} -w_{n_{k}}\| =0 $. From condition (D2), we have
	\begin{equation}\label{eq2}
		\left\|w_{n_{k}}-x_{n_{k}}\right\|=\delta_{n_{k}}\left\|x_{n_{k}}-x_{n_{k-1}}\right\| \rightarrow 0.
	\end{equation}
	Since $\{x_{n_{k}}\}$ is bounded,  there  exists a subsequence  $\{x_{n_{k_{j}}}\}$ of $\left\{x_{n_{k}}\right\}$ such that $x_{n_{k_{j}}} \rightharpoonup \bar{x}$ and
	$ \lim_{k \rightarrow \infty} \sup \left\langle u-p, x_{n_{k}}-p\right\rangle = \lim_{j \rightarrow \infty}\langle u-p, x_{n_{k_{j}}}-p\rangle $. By \eqref{eq2}, we have $w_{n_{k_{j}}} \rightharpoonup \bar{x}$. Using Lemma \ref{lem:fix}, we get $\bar{x} \in \operatorname{Fix}(T)$. Combining the projection property and $p=P_{\operatorname{Fix}(T)} u$, we obtain
	\begin{equation}\label{eqq}
		\lim_{k \rightarrow \infty} \sup \left\langle u-p, x_{n_{k}}-p\right\rangle = \lim_{j \rightarrow \infty}\langle u-p, x_{n_{k_{j}}}-p\rangle = \langle u-p, \bar{x}-p\rangle \le 0.
	\end{equation}
	From \eqref{MIMHA}, we obtain $ \|y_{n_{k}} -w_{n_{k}}\| = (1-\psi_{n_{k}})\|Tw_{n_{k}} -w_{n_{k}}\| \rightarrow 0 $. This together with \eqref{eq2}, we get $ \|y_{n_{k}}-x_{n_{k}}\| \le \|y_{n_{k}}-w_{n_{k}}\| + \|w_{n_{k}}-x_{n_{k}}\| \rightarrow 0 $. Further, combining condition (D1), we obtain
	\begin{equation}\label{eqd}
		\left\|x_{n_{k+1}}-x_{n_{k}}\right\| \le \nu_{n_{k}}\left\|u-x_{n_{k}}\right\|+\left(1-\nu_{n_{k}}\right)\left\|y_{n_{k}}-x_{n_{k}}\right\| \rightarrow 0.
	\end{equation}
	Combining \eqref{eqq} and \eqref{eqd}, we get that $ \limsup_{k \rightarrow \infty}\langle u-p,x_{n_{k+1}} -p\rangle \le 0 $, this together with condition (D2) implies that $\limsup_{k \rightarrow \infty} \sigma_{n_{k}} \leq 0$. From  Lemma \ref{lemma1} we observe that $\lim _{n \rightarrow \infty} S_{n}=0$ and hence $x_{n} \rightarrow p$ as $n \rightarrow \infty$. This completes the proof.
\end{proof}

\begin{remark}
	If $f: C \rightarrow C$ is a contractive mapping and we replace $ u $ by $ f(x_{n}) $ in \eqref{MIMHA}, we can obtain the following  viscosity iteration algorithm, for more details, see \cite{suzuki2007moudafi}.
\end{remark}

\begin{theorem}
	Let $ C $ be a nonempty closed convex subset of a real Hilbert space $ H $ and let $T: C \rightarrow C$  be a nonexpansive mapping with  $\operatorname{ Fix }(T) \neq \emptyset$. Let $f: C \rightarrow C$ be a $\rho$-contraction with $\rho \in[0,1)$, that is $ \|f(x)-f(y)\|\le \rho \|x-y\|, \forall x,y\in C $. Given two sequences $\left\{\psi_{n}\right\}$ and $\left\{\nu_{n}\right\}$ in $(0,1)$, the following conditions are satisfied:
	\begin{enumerate}[(K1)]
		\item $\lim_{n \rightarrow \infty}\nu_{n} = 0$ and  $\sum_{n=0}^{\infty} \nu_{n}=\infty$;
		\item $ \lim_{n \rightarrow\infty} \frac{\delta_{n}}{\nu_{n}}\left\|x_{n}-x_{n-1}\right\|=0 $.
	\end{enumerate}
	Set $ x_{-1},x_{0}\in C $ be arbitarily. Define a sequence $\left\{x_{n}\right\}$ by the following algorithm:
	\begin{equation}\label{MIMVA}
		\left\{\begin{array}{ll} w_{n}=x_{n} + \delta_{n}(x_{n}-x_{n-1}), \\ {y_{n}=\psi_{n} w_{n}+\left(1-\psi_{n}\right) T w_{n},}  \\ {x_{n+1}=\nu_{n} f(x_{n})+\left(1-\nu_{n}\right) y_{n}}, n \ge 0.\end{array}\right.
	\end{equation}
	Then the iterative sequence $\left\{x_{n}\right\}$ defined by \eqref{MIMVA} converges strongly to  $ z=P_{\operatorname{Fix}(T)} f(z) $.
\end{theorem}

\begin{remark}
	\begin{enumerate}[(i)]
		\item For special choice, the parameter $ \delta_{n} $ in the Algorithm \eqref{MIMHA} and the Algorithm \eqref{MIMVA} can be chosen the following:
		\begin{equation}\label{alpha}
			0 \leq \delta_{n} \leq \bar{\delta}_{n},\quad \bar{\delta}_{n}=\left\{\begin{array}{ll}{\min \left\{\frac{\xi_{n}}{\left\|x_{n}-x_{n-1}\right\|}, \frac{n-1}{n+\eta-1}\right\}}, & {\text { if } x_{n} \neq x_{n-1}}, \\ {\frac{n-1}{n+\eta-1}}, & {\text { otherwise}},\end{array}\right.
		\end{equation}
		for some $\eta \geq 3$ and $\left\{\xi_{n}\right\}$ is a positive sequence such that $\lim_{n \rightarrow \infty}\frac{\xi_{n}}{\nu_{n}}=0$. This idea derives from the recent inertial extrapolated step introduced in \cite{beck2009fast,attouch2016rate}.
		\item If $ \delta_{n}=0 $ for all $ n\ge 0 $, in the Algorithm \eqref{MIMHA} and  the Algorithm \eqref{MIMVA}, then we obtained the results of proposed by Kim and Xu \cite{kim2005strong} and Yao, Chen and Yao \cite{yao2008strong}, respectively.
	\end{enumerate}
\end{remark}

\section{Numerical experiments}\label{sec4}
In this section, we provide some numerical examples to illustrate the computational performance of the proposed algorithms. All the programs are performed in MATLAB2018a on a PC Desktop Intel(R) Core(TM) i5-8250U CPU @ 1.60GHz 1.800 GHz, RAM 8.00 GB.

\noindent \textbf{Example 3.1.}
	Let $H_{1}$ and $H_{2}$ be real Hilbert spaces and $T: H_{1} \rightarrow H_{2}$ a bounded linear operator. Let $ C $ and $ Q $ be nonempty closed and convex subsets of  $H_{1}$ and $H_{2}$, respectively. We consider the following  split feasibility problem (in short, SFP):
	\begin{equation}\label{P}
		\text { find } x^{*} \in C \text { such that } T x^{*} \in Q.
	\end{equation}
	For any $ f, g \in L^{2}([0,2 \pi]) $, we consider $H_{1}=H_{2}=L^{2}([0,2 \pi])$ with the inner product $ \langle f, g\rangle :=\int_{0}^{2 \pi} f(t) g(t) d t $ and the induced norm $ \|f\|_{2}:=\left(\int_{0}^{2 \pi}|f(t)|^{2} d t\right)^{\frac{1}{2}}$. Consider the half-space
	\[
	C=\left\{x \in L_{2}([0,2 \pi]) | \int_{0}^{2 \pi} x(t) d t \leq 1\right\},
	\text{ and }
	Q=\left\{x \in L_{2}([0,2 \pi])|\int_{0}^{2 \pi}\left| x(t)-\sin (t)\right|^{2} d t \leq 16\right\}.
	\]
	The set $ C $ and $ Q $ are nonempty closed and convex subsets of $L^{2}([0,2 \pi])$. Assume that $T: L^{2}([0,2 \pi]) \rightarrow L^{2}([0,2 \pi])$  is a bounded linear operator with its adjoint $ T $, it is defined by $(Tx)(t):=x(t)$. Then $\left(T^{*} x\right)(t)=x(t)$  and $\|T\|=1$. Therefore, \eqref{P} is actually a convex feasibility problem: $ \text {find } x^{*} \in C \cap Q $. Moreover, observe that the solution set of \eqref{P} is nonempty since  $x(t)=0$ is a solution.  For solving the \eqref{P}, Byrne \cite{byrne2002iterative} proposed the following algorithm:
	\[
	x_{n+1}=P_{C}\left(x_{n}-\lambda T^{*}\left(I-P_{Q}\right) T x_{n}\right),
	\]
	where $0<\lambda<2 L$ with Lipschitz constant $L=1 /\|T\|^{2}$. For the purpose of our numerical computation, we use the following formula for the projections onto $ C $ and  $ Q $, respectively, see \cite{bauschke2011convex}.
	\[
	P_{C}(x)=\left\{\begin{array}{ll}{\frac{1-a}{4 \pi^{2}} +x}, & {a>1}, \\ {x}, & {a \leq 1}.\end{array}\right.
	\text{ and }
	P_{Q}(x)=\left\{\begin{array}{ll}{\sin(\cdot) +\frac{4(x-\sin(\cdot) )}{\sqrt{b}}}, & {b>16}, \\ {x}, & {b \leq 16},\end{array}\right.
	\]
	where $a=\int_{0}^{2 \pi} x(t) d t$ and $b=\int_{0}^{2 \pi}|x(t)-\sin (t)|^{2} d t$. We consider different initial points $ x_{-1}= x_{0} $ and use  the stopping criterion
	\[
	E_{n}=\frac{1}{2}\left\|P_{C}x_{n}-x_{n}\right\|_{2}^{2}+\frac{1}{2}\left\|P_{Q}Tx_{n}-Tx_{n}\right\|_{2}^{2} <\epsilon.
	\]
	
	We use the modified  Mann Halpern algorithm (MMHA, i.e., MIMHA with $ \delta_{n}=0 $) \cite{kim2005strong}, the modified inertial Mann Halpern algorithm \eqref{MIMHA} (MIMHA), the modified  Mann viscosity algorithm (MMVA, i.e., MIMVA with $ \delta_{n}=0 $) \cite{yao2008strong} and the modified inertial Mann viscosity algorithm \eqref{MIMVA} (MIMVA) to solve Example 3.1. In all algorithms, set $ \epsilon =10^{-3}$, $ \psi_{n}=\frac{1}{100(n+1)^2} $, $ \nu_{n} = \frac{1}{n+1} $, $ \lambda = 0.25 $. In  MIMHA algorithm and  MIMVA algorithm, update $ \delta_{n} $ by \eqref{alpha} with $ \xi_{n} = \frac{10}{(n+1)^2} $ and $ \eta = 4 $. Set $ u=0.9x_{0} $ in the MIMHA algorithm and $ f(x)=0.9x_{n} $ in the MIMVA algorithm, respectively. Numerical results  are reported in Table \ref{tab1} and Fig. \ref{fig1}.	 In Table \ref{tab1}, ``Iter." and ``Time(s)" denote the number of iterations  and the cpu time in seconds, respectively.
	\begin{table}[htbp]
		\centering
		\caption{Computation results for Example 3.1}
			\begin{tabular}{lllllllllllll}
				\toprule
				&       & \multicolumn{2}{c}{MMHA} &       & \multicolumn{2}{c}{MIMHA} & & \multicolumn{2}{c}{MMVA} & & \multicolumn{2}{c}{MIMVA} \\
				\cmidrule{3-4}\cmidrule{6-7}\cmidrule{9-10} \cmidrule{12-13}   Cases & Initial points & Iter. & Time(s)   &       & Iter. & Time(s) &       & Iter. & Time(s) &       & Iter. & Time(s) \\
				\midrule
				I     & $  x_{0}=\frac{t^{2}}{10}$ &58       & 15.62      &       &56       &16.83   &       & 14      &3.64 & &8 &2.35 \\
				II    & $x_{0}=\frac{e^{\frac{t}{2}}}{3}$      &195       & 53.20     &       & 194      &60.40  &       &16       &4.34 & &10 &2.90 \\
				III   & $x_{0}=\frac{2^{t}}{16} $      &47       &12.85     &       &43       &13.03  &       &13       &3.36 & &8 &2.35 \\
				IV    &  $x_{0}=3\sin(2t)$     & 98      & 26.61      &       &95       &28.54  &       & 8      &2.08 & &5 &1.47 \\
				\bottomrule
			\end{tabular}%
		\label{tab1}%
	\end{table}%
	
	\begin{figure}[htbp]
		\centering  
		\subfigure[Case I]{
			\label{Fig1sub1}
			\includegraphics[width=0.24\textwidth]{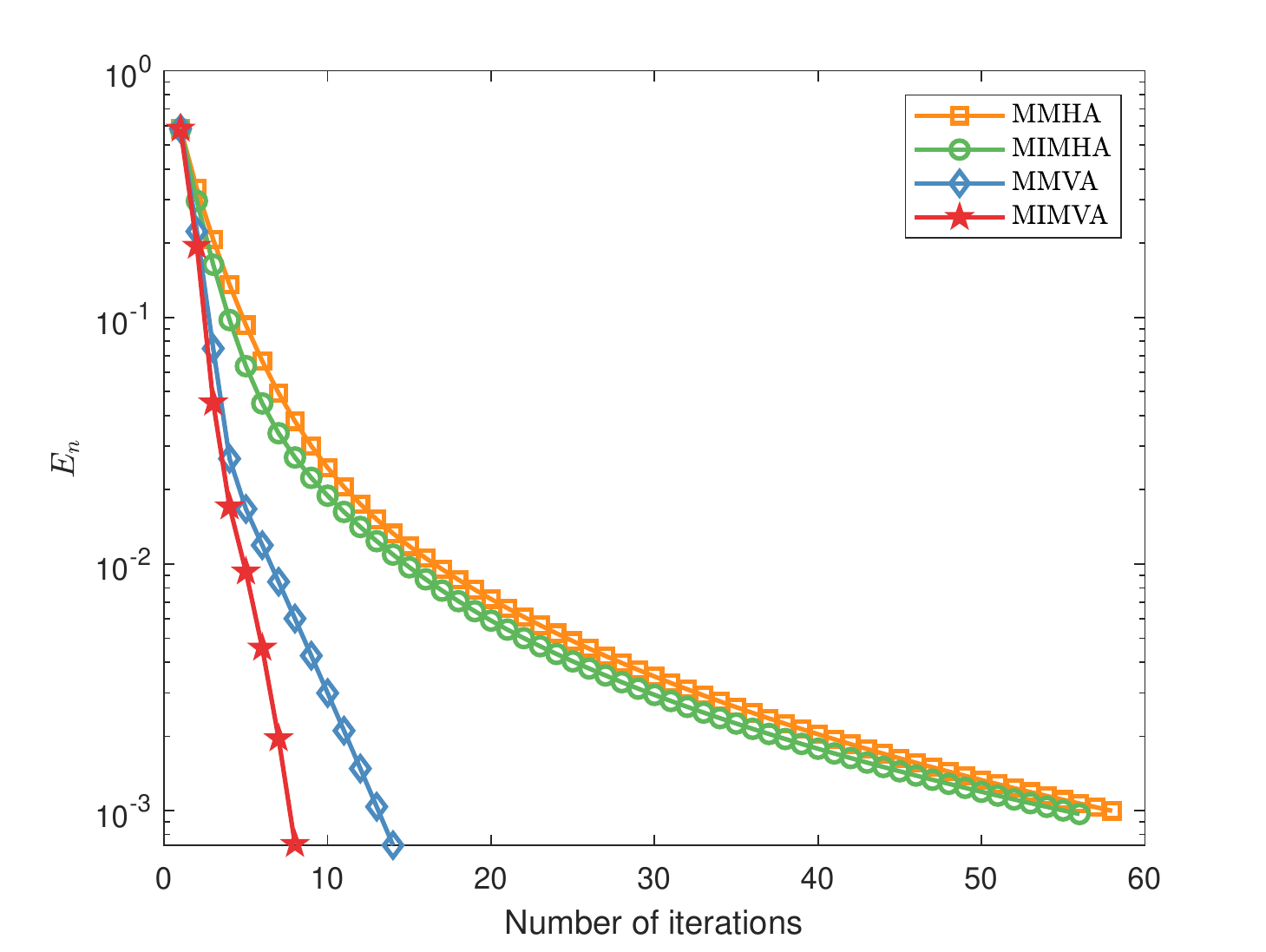}}
		\subfigure[Case II]{
			\label{Fig1sub2}
			\includegraphics[width=0.24\textwidth]{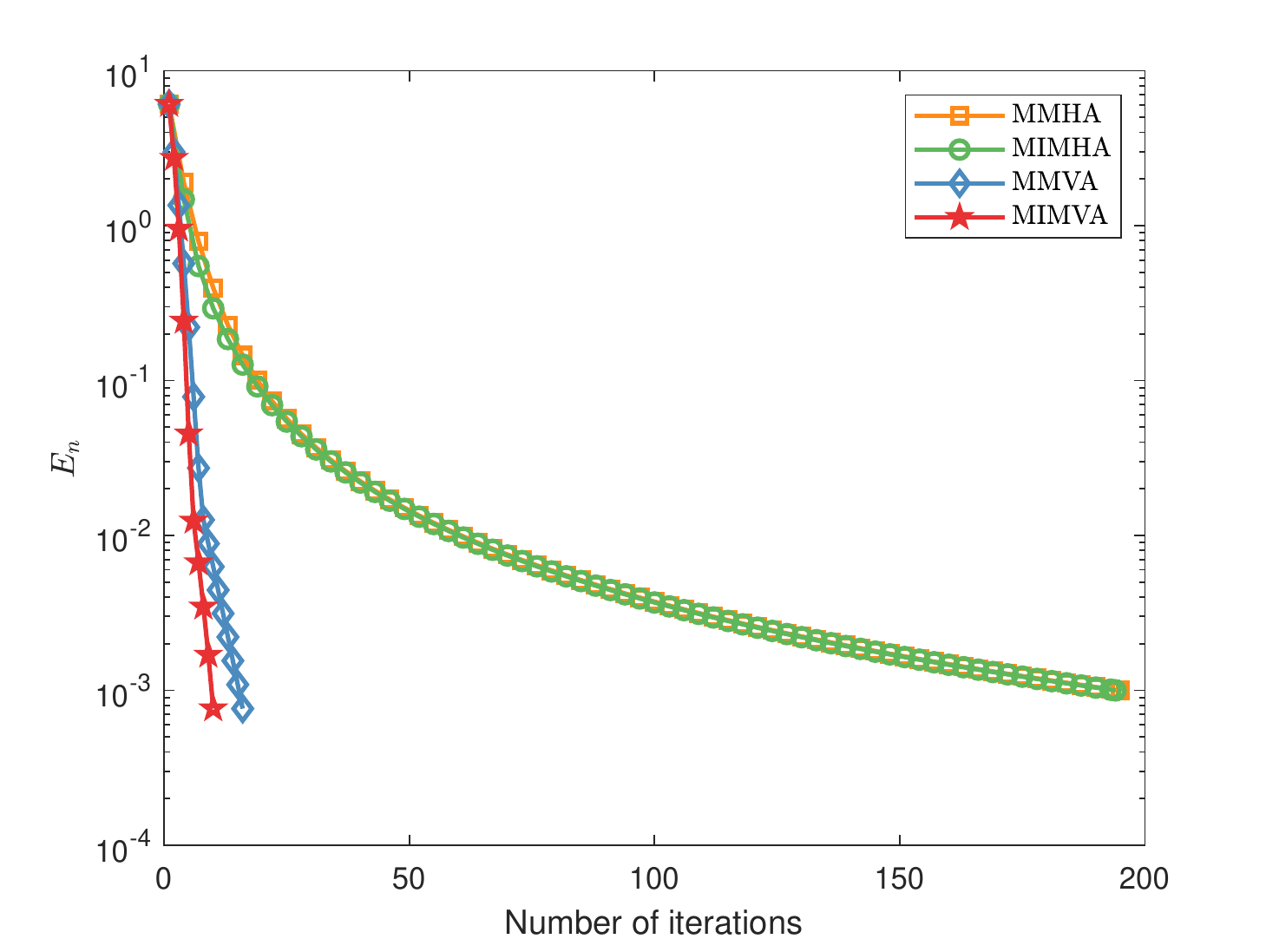}}
		\subfigure[Case III ]{
			\label{Fig1sub3}
			\includegraphics[width=0.24\textwidth]{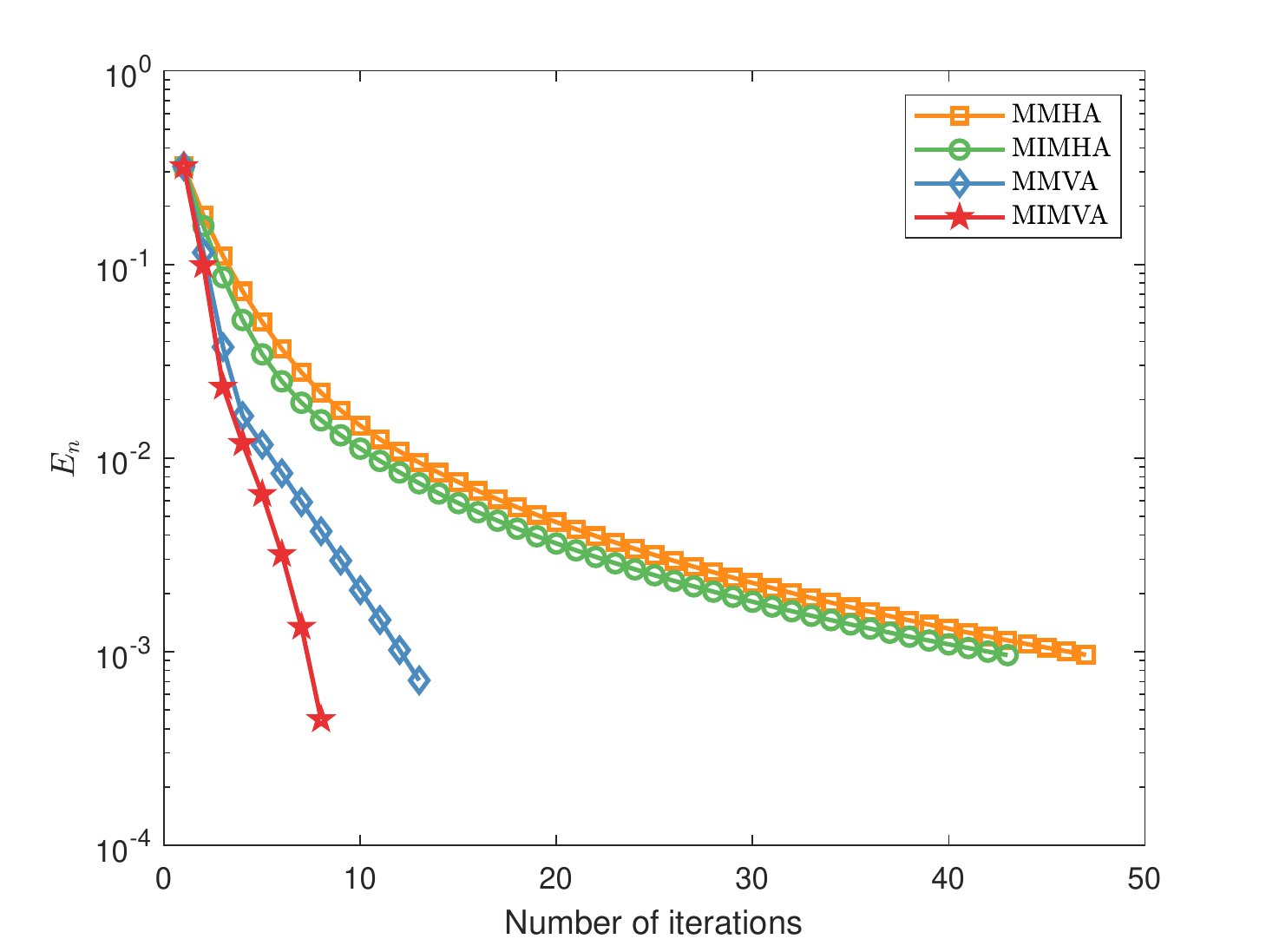}}
		\subfigure[Case IV ]{
			\label{Fig1sub4}
			\includegraphics[width=0.24\textwidth]{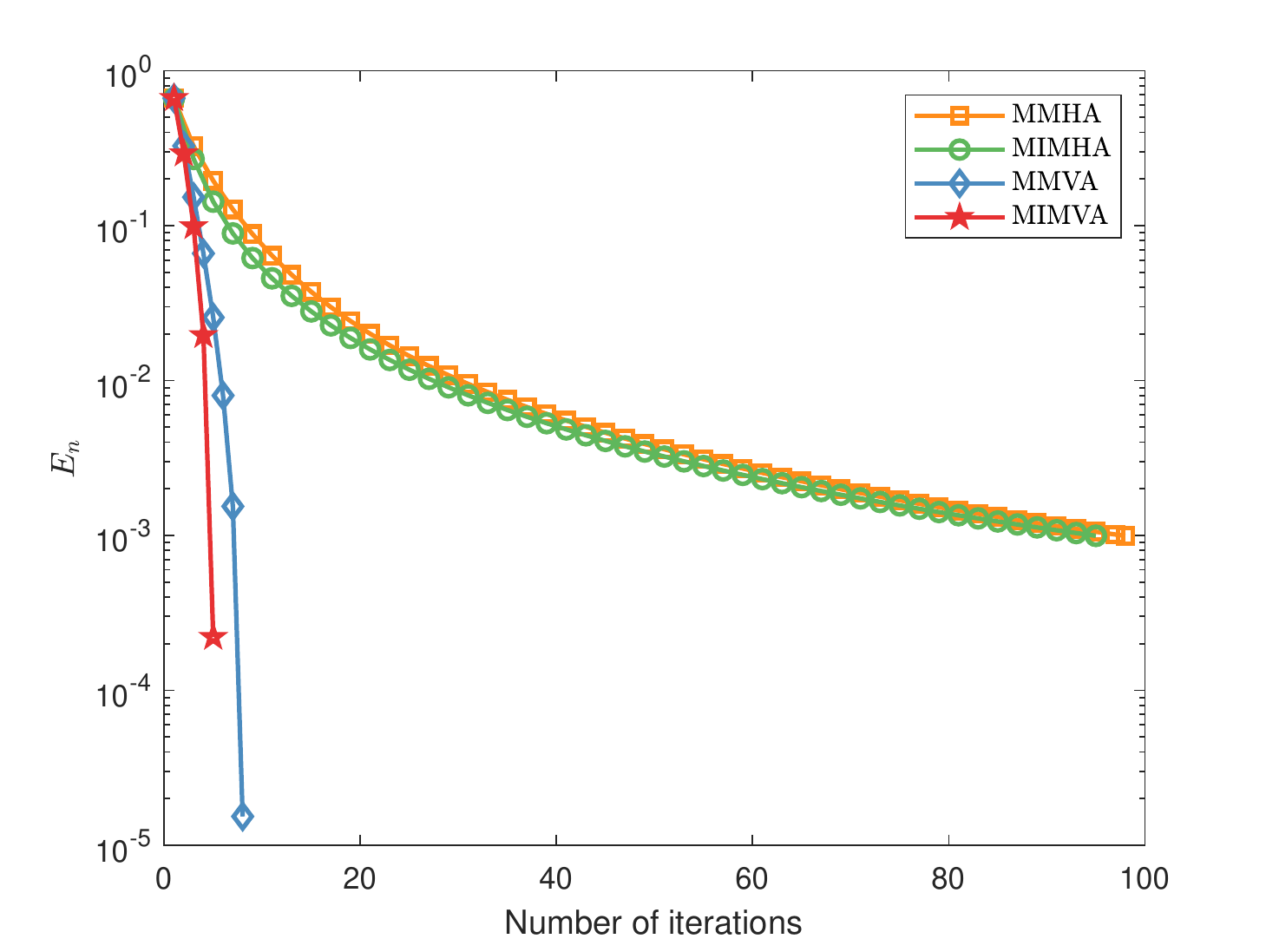}}
		\caption{Convergence behavior of iteration error $\{E_{n} \}$ for Example 3.1}
		\label{fig1}
	\end{figure}

\noindent \textbf{Example 3.2.}
	We consider the convex feasibility problem, for any nonempty closed convex set $C_{i} \subset {R}^{N}$ ($ i = 0,1, \ldots, m $), find
	$ x^{*} \in C:=\bigcap_{i=0}^{m} C_{i} $, where one assumes that $C \neq \emptyset$. Define a mapping $T: {R}^{N} \rightarrow {R}^{N}$ by $ T:=P_{0}\left(\frac{1}{m} \sum_{i=1}^{m} P_{i}\right) $, where $P_{i}=P_{C_{i}}$  stands for the  metric projection onto $ C_{i} $. Since $P_{i}$ is nonexpansive and hence the mapping $ T $ is also nonexpansive. Moreover, we find that
	$ \operatorname{ Fix }(T)=\operatorname{ Fix }\left(P_{0}\right) \bigcap_{i=1}^{m} \operatorname{ Fix }\left(P_{i}\right)=C_{0} \bigcap_{i=1}^{m} C_{i}=C $. In this experiment, we set $C_{i}$ as a closed ball with center $c_{i} \in {R}^{N}$ and radius $r_{i}>0$.  Thus $P_{i}$  can be computed with
	\[
	P_{i}(x):=\left\{\begin{array}{ll}{c_{i}+\frac{r_{i}}{\left\|c_{i}-x\right\|}\left(x-c_{i}\right)} & {\text { if }\left\|c_{i}-x\right\|>r_{i}}, \\ {x} & {\text { if }\left\|c_{i}-x\right\| \leq r_{i}}.\end{array}\right.
	\]
	Choose $r_{i}=1(i=0,1, \ldots, m)$, $c_{0}=[0,0,\ldots,0]$, $c_{1}=[1,0, \ldots, 0]$, and  $c_{2}=[-1,0, \ldots, 0]$. $c_{i} \in(-1 / \sqrt{N}, 1 / \sqrt{N})^{N}(i=3, \ldots, m)$ are randomly chosen. From the choice of  $c_{1}, c_{2}$ and $r_{1}, r_{2}$, we get that $\operatorname{ Fix }(T)=\{0\}$. Denote $E_{n}=\left\|x_{n}\right\|_{\infty}$ the iteration error of the algorithms.
	
	We use th CQ algorithm \eqref{cq} (CQ) \cite{nakajo2003strong}, the  inertial Mann algorithm \eqref{imann} (iMann) \cite{mainge2008convergence}, the modified  inertial Mann algorithm (MIMA) \cite{dong2018modified}, the modified  Mann viscosity algorithm (MMVA, i.e., MIMVA with $ \delta_{n}=0 $) \cite{yao2008strong} and the modified inertial Mann viscosity algorithm \eqref{MIMVA} (MIMVA) to solve Example 3.2. In all algorithms, set $ N=30 ,m=30$. Set $ \psi_{n}=\frac{1}{n+1} $ in the CQ algorithm and $ \delta_{n}=0.5 $, $ \psi_{n}=\frac{1}{n+1} $ in the iMann algorithm, respectively. Set $ \alpha_{n}=0.9 $, $ \lambda=1 $, $ \beta_{n}=\frac{1}{(n+1)^2} $, $ \mu=1 $ and $ \gamma_{n}=0.1 $ in the MIMA algorithm and $ \xi_{n}=\frac{10}{(n+1)^2} $, $ \eta=4 $, $ \psi_{n}=\frac{1}{100(n+1)^2} $, $ \nu_{n}=\frac{1}{n+1} $ and $ f(x)=0.1x_{n} $ in the MIMVA algorithm, respectively. Take maximum iteration of 1000 as a common stopping criterion. The initial values are randomly generated by the MATLAB function $ 10 \times$rand. Numerical results are reported in Fig. \ref{fig2}.
	\begin{figure}[htbp]
		\centering
		\subfigure[Convergence behavior of iteration error $\{E_{n} \}$]{
			\label{Fig2sub1}
			\includegraphics[width=0.4\textwidth]{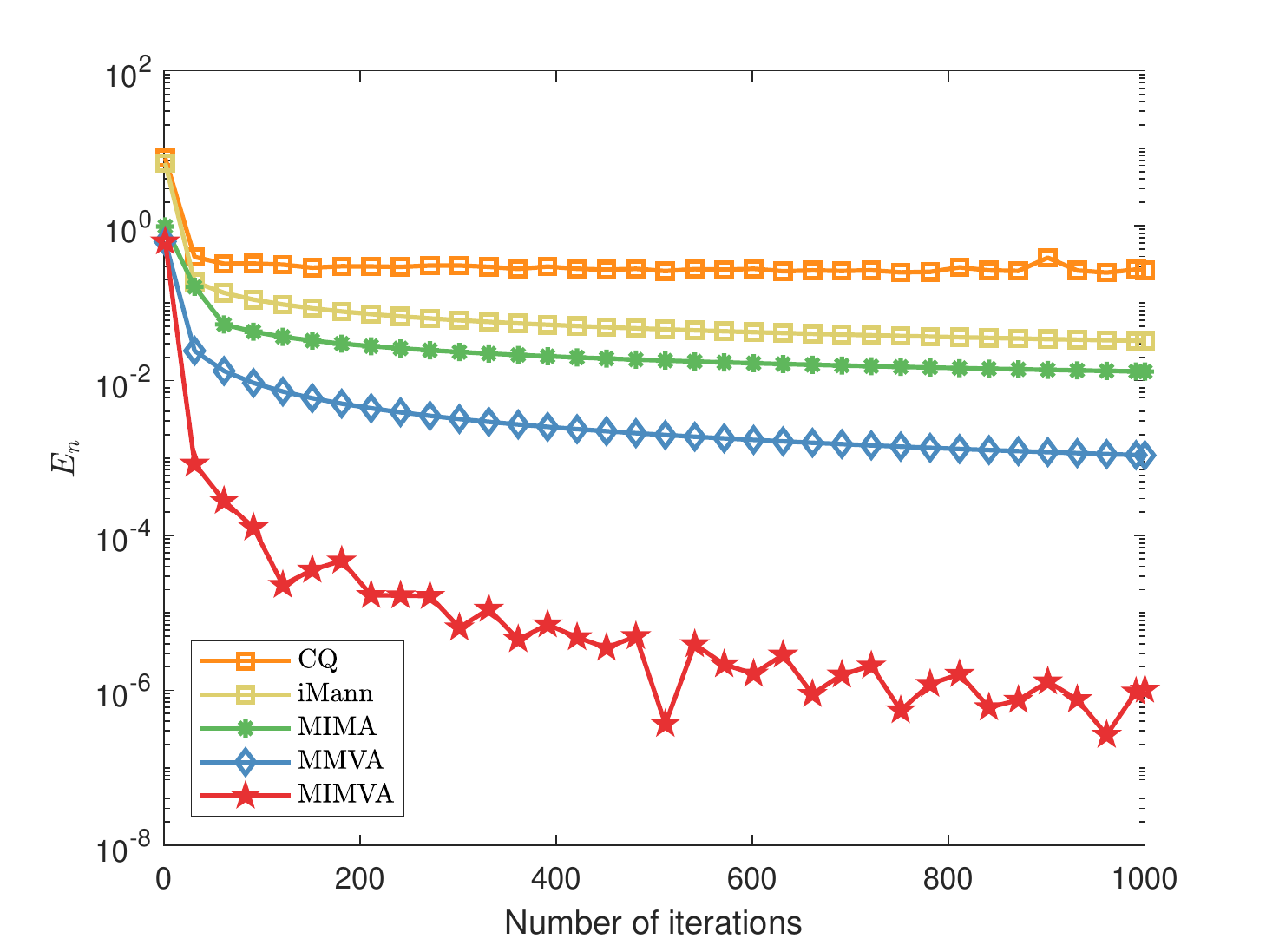}}
		\subfigure[MIMVA algorithm with different $ \eta $]{
			\label{Fig2sub2}
			\includegraphics[width=0.4\textwidth]{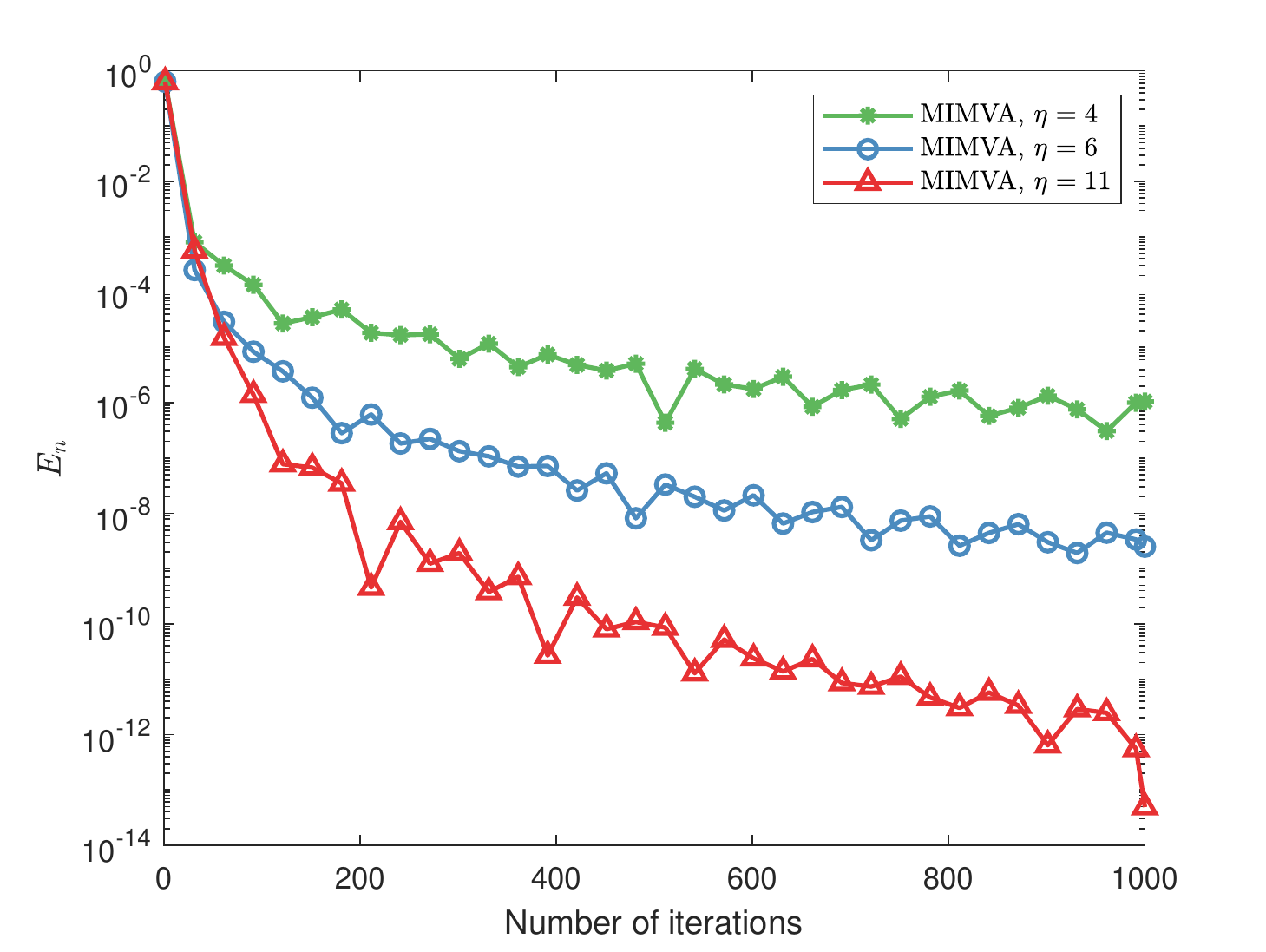}}
		\caption{Numerical results for  Example 3.2}
		\label{fig2}
	\end{figure}

\noindent \textbf{Example 3.3.}
	The Fermat–Weber problem is a famous model in location theory, which described as follows: find $ x\in R^{n} $ that solves
	\begin{equation}\label{fun}
		\min f(x):=\sum_{i=1}^{m} \omega_{i}\left\|x-a_{i}\right\|_{2},
	\end{equation}
	where $\omega_{i}>0$ are given weights and $a_{i} \in R^n$ are anchor points.  We know that the objective function $ f $ in \eqref{fun} is convex and coercive and hence the problem  has a nonempty solution set. It should be noted that $ f $ is not differentiable at the anchor points. The most famous method to solve the problem \eqref{fun} is Weiszfeld algorithm, see \cite{beck2015weiszfeld} for more discussion.   Weiszfeld proposed the following fixed point algorithm: $x_{n+1}=T\left(x_{n}\right), n\in N$. The mapping $T:{R}^{n} \backslash {A} \longmapsto {R}^{n}$ is defined by $ T(x):=\frac{1}{\sum_{i=1}^{m} \frac{\omega_{i}}{\left\|x-a_{i}\right\|}} \sum_{i=1}^{m} \frac{\omega_{i} a_{i}}{\left\|x-a_{i}\right\|} $, where ${A}=\left\{a_{1}, a_{2}, \ldots, a_{m}\right\}$. We consider a small example with $ n=3,m=8 $ anchor points,
	\[
	A=\left(\begin{array}{cccccccc}{0} & {10} & {0} & {10} & {0} & {10} & {0} & {10} \\ {0} & {0} & {10} & {10} & {0} & {0} & {10} & {10} \\ {0} & {0} & {0} & {0} & {10} & {10} & {10} & {10}\end{array}\right),
	\]
	and $\omega_{i}=1$ for all $i$. From the special selection of anchor points $ a_{i} (i=1,2,3,\cdots,8) $, we know that the optimal value of \eqref{fun} is $ x^* = (5, 5,5)^{\top} $. Fig. \ref{fig3} shows a schematic diagram of the anchor points and the optimal solution.
	\begin{figure}[h]
		\centering
		\includegraphics[width=0.4\linewidth]{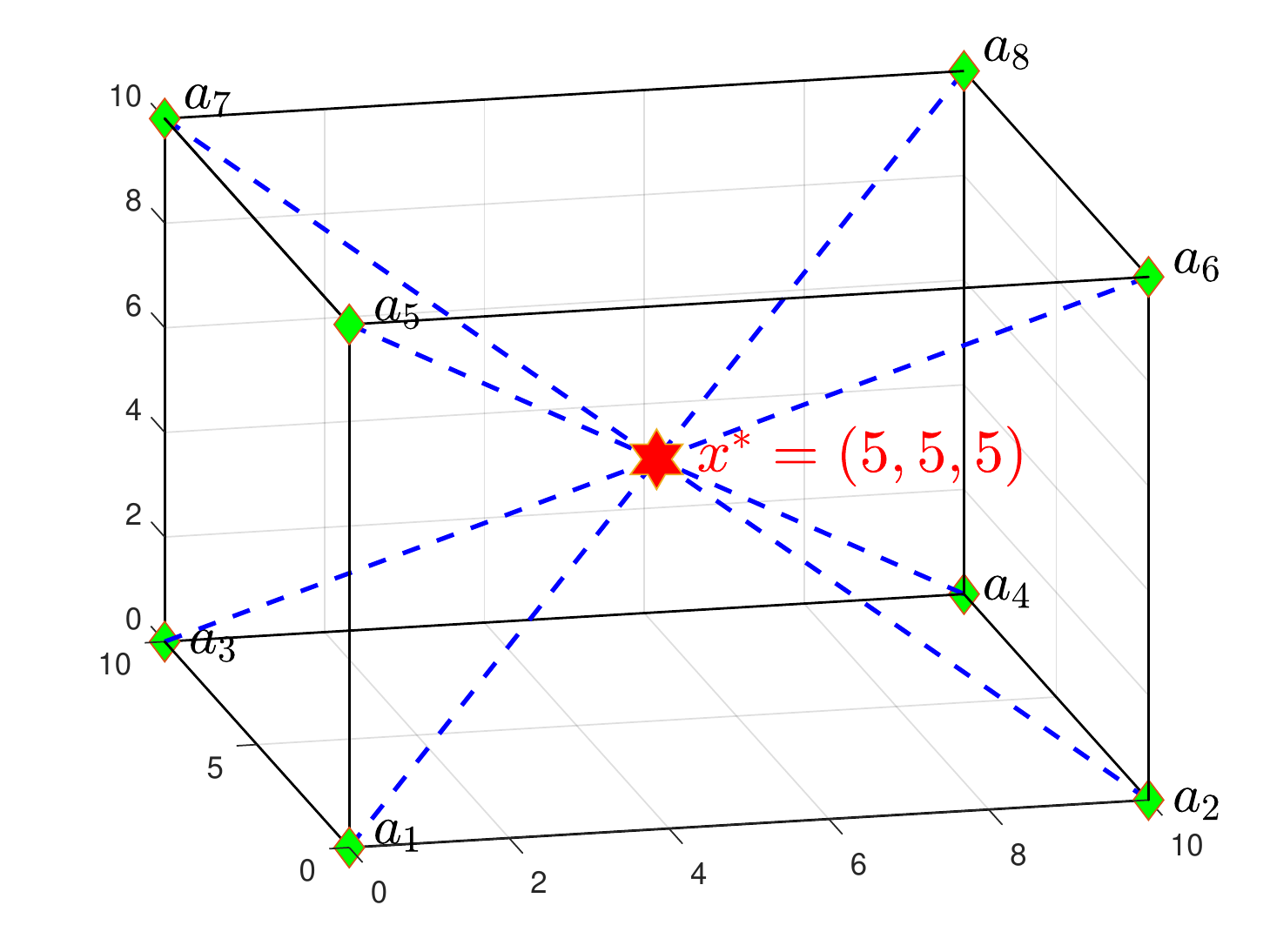}
		\caption{Schematic diagram of anchor points and optimal solution for Example 3.3}
		\label{fig3}
	\end{figure}
	
	We use the modified inertial Mann Halpern algorithm \eqref{MIMHA} (MIMHA) and the modified inertial Mann viscosity algorithm \eqref{MIMVA} (MIMVA) to solve Example 3.3. In MIMHA algorithm  and MIMVA algorithm, set $ \xi_{n}=\frac{10}{(n+1)^2} $, $ \eta=4 $, $ \psi_{n}=\frac{1}{100(n+1)^2} $, $ \nu_{n}=\frac{1}{n+1} $.  Set $ u=0.9x_{0} $ in the MIMHA algorithm and $ f(x)=0.9x_{n} $ in the MIMVA algorithm, respectively. Take $ E_{n}=\|x_{n}-x^*\|_{2} $ as iteration error of the algorithms and maximum iteration of 1000 as a common stopping criterion. The initial values are randomly generated by the MATLAB function $ 10 \times$rand. Numerical results are reported in Fig. \ref{fig4}.
	\begin{figure}[htbp]
		\centering  
		\subfigure[Convergence process of iterative sequence $ \{x_{n}\} $]{
			\label{Fig7sub1}
			\includegraphics[width=0.4\textwidth]{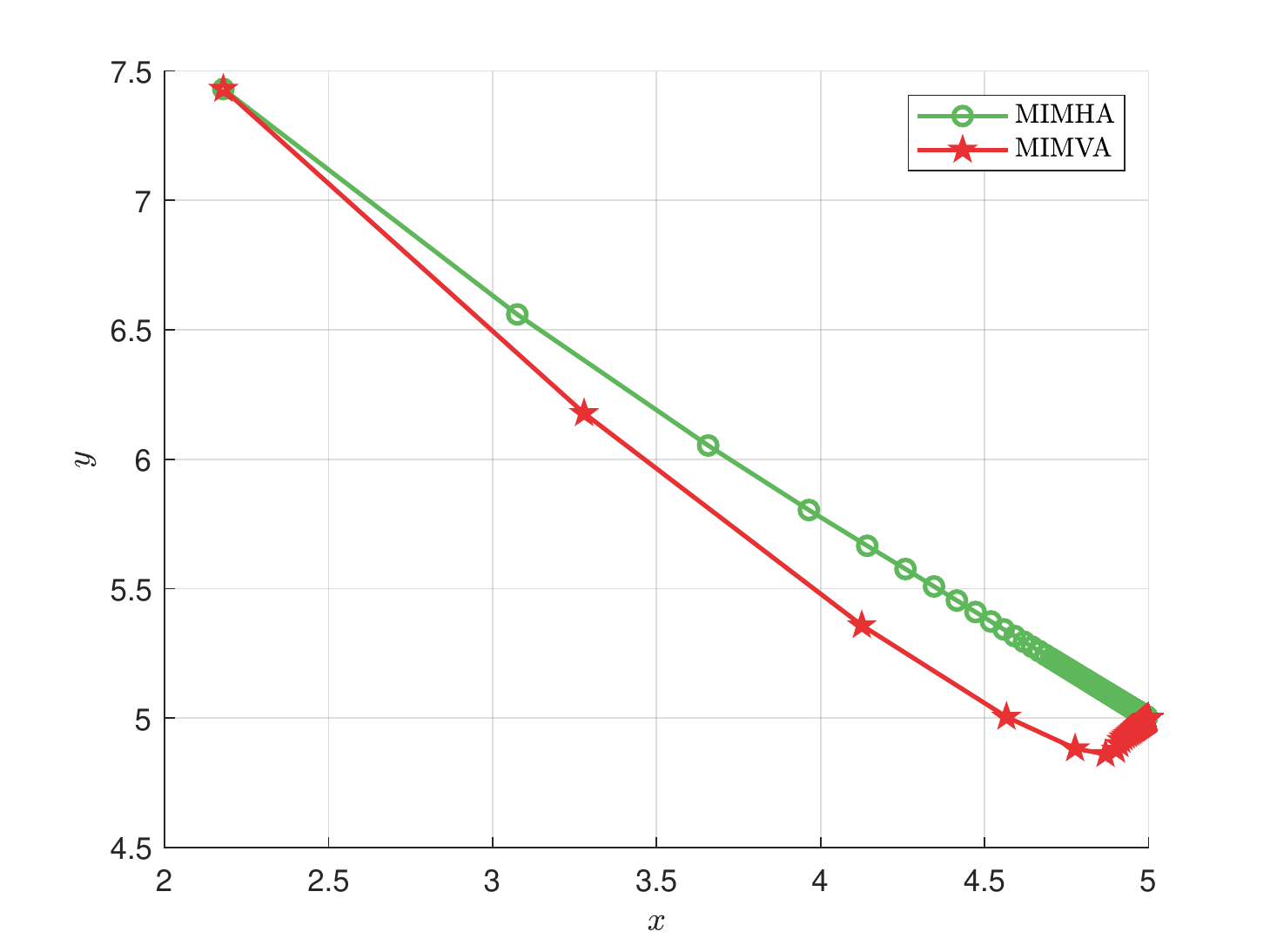}}
		\subfigure[Convergence behavior of iteration error $\{E_{n} \}$]{
			\label{Fig7sub2}
			\includegraphics[width=0.4\textwidth]{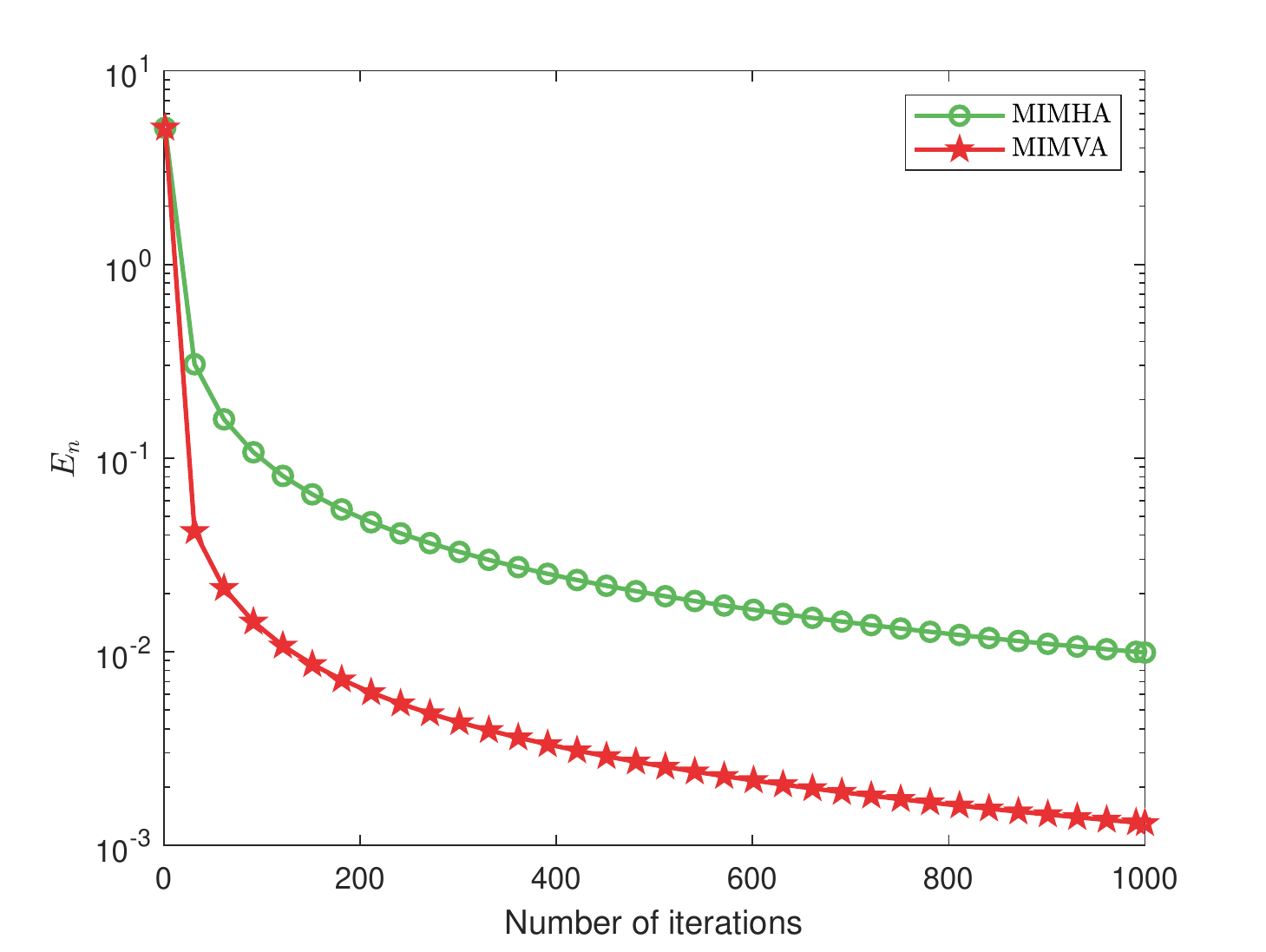}}
		\caption{Convergence behavior of $ \{x_{n}\} $ and $ \{E_{n}\} $ for Example 3.3}
		\label{fig4}
	\end{figure}

\begin{remark}
	\begin{enumerate}[(i)]
		\item From Example 3.1--Example 3.3, we observe that Algorithm \eqref{MIMVA} is efficient, easy to implement, and most importantly very fast. In addition, the inertial parameter \eqref{alpha} can significantly improve the convergence speed, see Fig. \ref{Fig2sub2}.
		\item The Algorithm \eqref{MIMVA} proposed in this paper can improve some known results in the field, see Fig. \ref{Fig2sub1}. It should be noted that
		the choice of initial values does not affect the calculation performance of the algorithm, see Table \ref{tab1}.
	\end{enumerate}
\end{remark}
\section{Conclusions}\label{sec5}
This paper discussed the modified inertial Mann Halpern and viscosity algorithms. Strong convergence results are obtained under some suitable conditions.  Finally, our proposed algorithms are applied to split feasibility problem, convex feasibility problem and location theory. Note that the algorithms and results presented in this paper can summarize and improve some known results in the area.

\section*{References}

\end{document}